\documentclass[12pt]{amsart}
\usepackage{amssymb,latexsym,comment,url}

\usepackage{amsmath}
\usepackage{amsthm}
\usepackage{amssymb}  
\usepackage{latexsym} 
\usepackage[all]{xy}
\usepackage{comment}
\usepackage{rotating}
\usepackage{slashbox}
\usepackage{multirow}
\usepackage{rotating}
\usepackage{fancyhdr}
\usepackage{tikz}
\usepackage{subcaption}
\usetikzlibrary{calc}

\newenvironment{Proof}{\par\noindent{\sc Proof:}}%
                      {\hspace*{\fill}\nobreak$\Box$\par\medskip}
                       {\hspace*{\fill}\nobreak$\Box$\par\medskip}

\newtheorem{Theorem}{Theorem}[section]
\newtheorem{Proposition}[Theorem]{Proposition}
\newtheorem{Lemma}[Theorem]{Lemma}
\newtheorem{Corollary}[Theorem]{Corollary}

\theoremstyle{definition}
\newtheorem{Definition}[Theorem]{Definition}

\renewcommand{\baselinestretch}{1.1}

\begin{document}

\title[LDU Decomposition]%
{A Combinatorial Interpretation of the LDU Decomposition of Totally Positive Matrices }

\author{M.~El Gebali}
\address{American University in Cairo, Mathematics and Actuarial Science Department, AUC Avenue, New Cairo, Egypt}
\email{m.elgebali@aucegypt.edu}

\author{N.~El-Sissi}
\address{American University in Cairo, Mathematics and Actuarial Science Department, AUC Avenue, New Cairo, Egypt}
\email{nelsissi@aucegypt.edu}

\date{26th October 2015}

\begin{abstract}
  We study the combinatorial description of the LDU decomposition of totally positive matrices. We give a description of the lower triangular L, the diagonal D, and the upper triangular U matrices of the LDU decomposition of totally positive matrices in terms of the combinatorial structure of essential planar networks described by Zelvinsky and Fomin in \cite{Fomin}. Similarly, we find a combinatorial description of the inverses of these matrices. In addition, we provide recursive formulae for computing the L, D, and U matrices of a totally positive matrix.\end{abstract}

\maketitle


\section{Introduction}
\label{sec:intro}
The study of the class of totally positive matrices was initiated in the 1930s by F. R. Gantmacher and M. G. Krein~\cite{GantKrein}. Also, an extensive study of totally positive matrices is covered in S. Karlin's book~\cite{Karlin}. Totally positive  matrices are a class of matrices that is worth investigating not only because of its mathematical beauty, but also because of their myriad of applications. More specifically, they arise in many applications, to name a few,  statistics, approximation theory, operator theory, combinatorics, and planar resistor network~\cite{Fomin,GascaMicch,Karlin}. 

A matrix is said to be totally positive (totally nonnegative) if all its minors are positive (nonnegative). In this note, we focus our attention on totally positive matrices, but the results can be extended to totally nonnegative matrices. Factorizations of totally positive matrices has been extensively studied \cite{Canto, Cryer} to minimze the computation of minors while checking for total positivity. More importantly, in \cite{Cryer, Gant} it was proved that: 

\begin{Theorem}\label{uniqueLU}
If a square matrix $A$ has a nonzero principal minors, then $A$ has a unique $LU$ factorization such that $L$ has a unit diagonal. 
\end{Theorem}
This is in turn can be directly applied to totally positive matrices (and totally nonnegative matrices with non-zero principal minors).

In \cite{Fomin}, a combinatorial approach to study the parametrization of totally positive matrices has been adopted. More specifically, a bijection between totally positive matrices and essential planar networks is given in the following theorem:

\begin{Theorem}[Fomin and Zelvinsky]\label{121}
There exists a one-to-one correspondence between essential positive weightings of essential planar networks and totally positive matrices. 
\end{Theorem}

In the second section of the paper, we define positively weighted edges of essential planar sub-network to which we associate lower triangular, diagonal, and upper triangular matrices. We then combine results from \cite{Cryer}, \cite{Gant}, and \cite{Fomin} to show that there is a bijection between the lower triangular (L), the diagonal (D), and the upper triangular (U) matrices obtained from the LDU factorization of a totally positive matrix and their corresponding essential planar sub-networks. In addition, the concatenation of these subnetworks yields the same network described in~\cite{Fomin}. 

In section \ref{formulas}, we obtain recursive formulae for computing the $(n+1) \times (n+1)$ lower, $L_n$,and upper triangular, $U$, matrices described in section \ref{LDU}. We then provide closed form formulae for the lower and upper triangular matrices in light of their corresponding planar sub-networks. 

In the last section, we provide combinatorial description for computing the inverses of L, D, and U; and hence the inverse of a totally positive matrix. In addition, we obtain closed-form formulae for computing $L^{-1}$, $D^{-1}$ and $U^{-1}$.

\section{LDU Decomposition of a Totally Positive Matrix}\label{LDU}
  In this section we use the fact that any totally positive matrix can be decomposed into the L, D, and U matrices, \cite{Cryer}. More specifically, it was proved that if $A$ is a totally positive matrix, then $A$ has a unique LU-factorization where $L$ is a unitlower triangular and $U$ an upper triangular matrix, \cite{Cryer,Gant}. This in turn gives a unique LDU decomposition of the matrix $A$, where $L$ and $U$ are unit lower and unit upper triangular matrices, respectively, and $D$ a diagonal matrix. In addition, we use Theorem~\ref{121} to show that there is a bijection between the LDU factorization of $A$ and subgraphs of the essential planar network defined.


We define a planar network of order $n$ to be an acyclic, planar directed graph. We will assume that each network has $(n+1)-$sources and $(n+1)-$sinks and all edges are directed from left to right, where each edge $\pi$ is assigned a scalar weight $\omega(\pi)$. In addition, the weight of a path $p$ from source $i$ to sink $j$ is defined as the product of the weights of its edges, namely $\displaystyle \omega(p)=\prod_{\pi \in p}\omega(\pi).$

In our setting, we will consider three planar sub-networks, which we call the $L-$type, the $D-$type and the $U-$type networks. Throughout this note, $\mathbb{N}$ denotes the set of natural numbers, and $\mathbb{N}_0$ the set of nonnegative integers. 

The $L-$type subgraph of order $n$ is a planar network  consisting of fall steps and horizontal steps. More specifically, we give the following definition. 
 
\begin{Definition}\label{L_Path}
For every $n,m \in \mathbb{N}$ define an $\mathcal{L}^{m}_{n}$-path to be any sequence of $n+1$ points, 
$(x_{i},y_{i})\in\mathbb{N}_0\times\mathbb{N}_0$ for $i=0, 1, 2, \ldots, n$ with the following conditions:
\begin{enumerate}
\item
$x_{i}\,=\,m+i$ for all $i$.
\item
$y_{i} \in \{0, 1, 2, \ldots, n\}$ for all $i$.
\item
$y_{i+1} \in \{y_i, y_{i}-1\}$  for all $i < n$.
\item
If $y_{i+1}=y_{i}-1$, then $y_{i}\geq n-i$.
\end{enumerate}
Also, define $\mathcal{P}^{L_n}_{i,j}$ to be the set of all $\mathcal{L}^{0}_{n}$-paths with $y_{0}=i$ and $y_{n}=j$.
\end{Definition}

In addition, 

 \begin{Definition}\label{Ln}
For every $n\in \mathbb{N}$ define $L_n$ to be the $(n+1)\times(n+1)$ matrix such that \[\displaystyle L_n[i,j] =\sum_{p\in \mathcal{P}^{L_n}_{i,j}}\omega(p)\] for all $0\leq i, j\leq n$, where the weight of the edges are defined as follows,
\begin{itemize}
\item For $j\,\neq\,0$, $\omega([(m+i,j),(m+i+1,j-1)])=t_{j,i+j-n}$, 
\item
$\omega([(m+i,j),(m+i+1,j)]) =1$,
\end{itemize}
for all $i, j = 0, 1, \ldots, n.$
\end{Definition}
The following lemma is a consequence of definition~\ref{L_Path}.
\begin{Lemma} \label{y_dcrsng}
If $[(x_{i},y_{i})]_{i =0, \ldots, n}$ is an $\mathcal{L}^{m}_{n}$-Path, then $y_{i} \leq y_{j}$ whenever $i>j$.
\end{Lemma}
\begin{proof}
This follows directly from the fact that $y_{i+1}\,\leq\,y_{i}$ for all $i\,<\,n$.
\end{proof}

The next corollary shows that the matrix whose entries are defined in definition ~\ref{Ln} is unit lower triangular. 

\begin{Corollary}\label{Ln_unitlower}
For all $n \in\mathbb{N}$, the weight matrix of any $L-$type network of order $n$ is unit lower triangular. 
\end{Corollary}
\begin{proof}
If $i<j$, then it follows directly from Lemma~\ref{y_dcrsng} that $\mathcal{P}_{i,j}^{L_n}$ is empty, making the elements above the diagonal all $0$. To show that it has $1$'s on the diagonal, consider the case where $i = j$. Since we consider a decreasing sequence $y_0,\ldots,y_n$ with $y_0\,=\,y_n\,=i$ , and from Lemma~\ref{y_dcrsng}, $y_n\leq y_{_k}\leq y_0$ for all $k\in\{0.\ldots,n\}$. It follows that $y_{_k}=i$ and the set $\mathcal{P}^{L_n}_{i,i}$ is $\{[(k,i)]_{k=0, \ldots, n}\}$. 
It follows that the diagonal entry must be $1$ since the weight of that path is $1$. Consequently, the elements of the diagonal of the weight matrix are all $1$'s.
\end{proof}
 Figure~\ref{fig:Lexample} illustrates an $L-$type network of order $2$ with weights defined as in Definition~\ref{Ln} and Figure~\ref{fig:L2} is the weight matrix associated to it.
\begin{figure}[ht]
\begin{subfigure}{0.35\textwidth}
\centering
\begin{tikzpicture}
\draw[black, thick]   (0,0)node[anchor=east]{$0$}--(4,0)node[anchor=west]{$0$}--(3.5,0.25)node[anchor=south]{$t_{1,0}$}--(1.5,1.25)node[anchor=south]{$t_{2,0}$}--(0,2)node[anchor=east]{$2$}--(4,2)node[anchor=west]{$2$};
\draw[black, thick]   (0,1)node[anchor=east]{$1$}--(4,1)node[anchor=west]{$1$}--(3.5,1.25)node[anchor=south]{$t_{2,1}$}--(2,2);
\end{tikzpicture}
\caption{$L-$type network of order $2$}\label{fig:Lnet2}
\end{subfigure}
\begin{subfigure}{0.35\textwidth}
\centering
$\left(\begin{array}{ccc}
1               &   0               &   0\\
t_{1,0}         &   1               &   0\\
t_{2,0}t_{1,0}  &   t_{2,0}+t_{2,1} &   1
\end{array}\right)$
\caption{The weight matrix $L_2$}\label{fig:L2}
\end{subfigure}
\caption{Description of $L_2$}\label{fig:Lexample}
\end{figure}
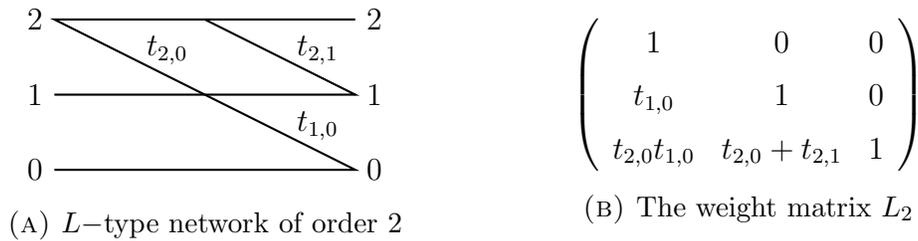

Similarly, we define another subgraph of order $n$, namely the $D-$type path. This path consists of horizontal steps only. 

\begin{Definition}\label{D_Path}
For every $n,m\in \mathbb{N},$ define a $\mathcal{D}^{m}_{n}$-Path to be any sequence of the form 
$[(m,j),(m+1,j)]$ for some $j=0, 1, \dots, n$.\\
Also, define $\mathcal{P}^{D_n}_{i,j}$ to be the set of all $\mathcal{D}^{0}_{n}$-paths with $y_{0}=i$ and $y_{1}\,=j$.
\end{Definition}

In analogy to the $L$-type path, we can associate a matrix $D$ to a $D-$path. 

\begin{Definition}\label{Dn}
For every $n\in\mathbb{N},$ define $D_n$ to be the $(n+1)\times(n+1)$ matrix such that \[D_n[i,j]=\sum_{p\in\mathcal{P}^{D_n}_{i,j}}\omega(p)\] for all $0\,\leq\,i,j\leq n$; where $\omega[(m, i), (m+1,i)]=t_{i,i}$.
\end{Definition}

\begin{Lemma}\label{Dn_diagonal}
For all $n \in \mathbb{N}$, $D_n$ is a diagonal matrix with diagonal elements $D_n[i,i] = t_{i,i}$. More specifically, 
\[D_n = \left(\begin{array}{cccc}
t_{0,0} & 0 & \cdots & 0\\
0 & t_{1,1} & \cdots & 0\\
\vdots & & \ddots & \vdots\\
0 & 0 & \cdots & t_{n,n}
\end{array}\right).\]
\end{Lemma}

\begin{proof}
If $i \neq j$, then $D_n[i,j]=0$ from definition~\ref{Dn}; thus $D_n$ is a diagonal matrix. Now, for $i=j$, the set $\mathcal{P}^{D_n}_{i,i}=\{[(0, i), (1, i)]\}$. It follows that $D_n[i,i]\,=\,t_{i,i}$.
\end{proof}

Finally, we can define the $U-$type network in which only rise and horizontal steps are allowed.
 
\begin{Definition}\label{U_Path}
For every $n,m \in \mathbb{N},$ define a $\mathcal{U}^{m}_{n}$-path to be any sequence of $n+1$ points 
$(x_{i},y_{i})\in\mathbb{N}_0\times\mathbb{N}_0$ for $i=0, 1, 2, \ldots, n$ with the following conditions:
\begin{enumerate}
\item
$x_{i}\,=\,m+i$ for all $i$.
\item
$y_{i} \in \{0, 1, 2, \ldots, n\}$ for all $i$.
\item
$y_{i+1} \in \{y_i, y_{i}+1\}$  for all $i < n$.
\item
If $y_{i+1}=y_{i}+1$, then $y_{i}\geq i$.
\end{enumerate}
Also, define $\mathcal{P}^{U_n}_{i,j}$ to be the set of all $\mathcal{U}^{0}_{n}$-paths with $y_{0}=i$ and $y_{n}=j$.
\end{Definition}

Then, we assign a matrix $U_n$ to the $U-$type network of order $n$.

 \begin{Definition}\label{Un}
For every $n\in \mathbb{N},$ define $U_n$ to be the $(n+1)\times(n+1)$ matrix such that \[\displaystyle U_n[i,j] =\sum_{p\in \mathcal{P}^{U_n}_{i,j}}\omega(p)\] for all $0\leq i, j\leq n$, where the weight of the edges are defined as follows:
\begin{itemize}
\item For $j\,\neq\,0$, $\omega([(m+i,j),(m+i+1,j+1)])=t_{j-i,j+1}$, 
\item 
$\omega([(m+i,j),(m+i+1,j)]) =1$,
\end{itemize}
for all $i, j = 0, 1, \ldots, n.$
\end{Definition}
The following lemma and corollary are similar to Lemma~\ref{y_dcrsng} and Corollary~\ref{Ln_unitlower}. The proofs are also similar. They can be replicated by reversing the inequalities.
\begin{Lemma} \label{y_incrsng}
If $[(x_{i},y_{i})]_{i =0, \ldots, n}$ is a $\mathcal{U}^{m}_{n}$-Path, then $y_{i} \geq y_{j}$ whenever $i>j$.
\end{Lemma}

\begin{Corollary}\label{Un_unitupper}
For all $n \in\mathbb{N}$, the weight matrix of a $U-$type network of order $n$ is unit upper triangular. 
\end{Corollary}

Figure~\ref{fig:subnetworks} illustrates an example of the essential planar subnetworks of order $2$ and their concatenation which yields its essential planar network of order $2$.

\begin{figure}[ht]
\begin{subfigure}[b]{0.3\textwidth}
\centering
\begin{tikzpicture}
\draw[black, thick]   (0,0)node[anchor=east]{$0$}--(4,0)node[anchor=west]{$0$}--(3.5,0.25)node[anchor=south]{$t_{1,0}$}--(1.5,1.25)node[anchor=south]{$t_{2,0}$}--(0,2)node[anchor=east]{$2$}--(4,2)node[anchor=west]{$2$};
\draw[black, thick]   (0,1)node[anchor=east]{$1$}--(4,1)node[anchor=west]{$1$}--(3.5,1.25)node[anchor=south]{$t_{2,1}$}--(2,2);
\end{tikzpicture}
\caption{$L-$type network of order $2$}
\end{subfigure}
\begin{subfigure}[b]{0.3\textwidth}
\centering
\newcommand{\dr}[1]{\draw[black, thick] (0,#1)node[anchor=east]{#1}--(1,#1)node[anchor=south]{$t_{#1,#1}$}--(2,#1)node[anchor=west]{#1};}
\begin{tikzpicture}
\dr{0}\dr{1}\dr{2}
\end{tikzpicture}
\caption{$D-$type network of order $2$}
\end{subfigure}
\begin{subfigure}[b]{0.3\textwidth}
\centering
\begin{tikzpicture}
\draw[black, thick]   (4,0)node[anchor=west]{$0$}--(0,0)node[anchor=east]{$0$}--(0.5,0.25)node[anchor=south]{$t_{0,1}$}--(2.5,1.25)node[anchor=south]{$t_{0,2}$}--(4,2)node[anchor=west]{$2$}--(0,2)node[anchor=east]{$2$};
\draw[black, thick]   (4,1)node[anchor=west]{$1$}--(0,1)node[anchor=east]{$1$}--(0.5,1.25)node[anchor=south]{$t_{1,2}$}--(2,2);
\end{tikzpicture}
\caption{$U-$type network of order $2$}
\end{subfigure}

\begin{subfigure}[b]{\textwidth}
\centering
\begin{tikzpicture}
\draw[black, thick]   (0,0)node[anchor=east]{$0$}--(5,0)node[anchor=south]{$t_{0,0}$}--(10,0)node[anchor=west]{$0$};
\draw[black, thick]   (0,1)node[anchor=east]{$1$}--(5,1)node[anchor=south]{$t_{1,1}$}--(10,1)node[anchor=west]{$1$};
\draw[black, thick]     (0,2)node[anchor=east]{$2$}--(5,2)node[anchor=south]{$t_{2,2}$}--(10,2)node[anchor=west]{$2$};
\draw[black, thick]   (0,2)--(1.5,1.25)node[anchor=south]{$t_{2,0}$}--(3.5,0.25)node[anchor=south]{$t_{1,0}$}--(4,0)--(6,0)--(6.5,0.25)node[anchor=south]{$t_{0,1}$}--(8.5,1.25)node[anchor=south]{$t_{0,2}$}--(10,2);
\draw[black, thick]   (2,2)--(3.5,1.25)node[anchor=south]{$t_{2,1}$}--(4,1)--(6,1)--(6.5,1.25)node[anchor=south]{$t_{1,2}$}--(8,2);
\end{tikzpicture}
\caption{Essential planar network of order $2$}
\label{fig:essplanar}
\end{subfigure}
\caption{Concatenation of essential planar subnetworks of order $2$}\label{fig:subnetworks}
\end{figure}
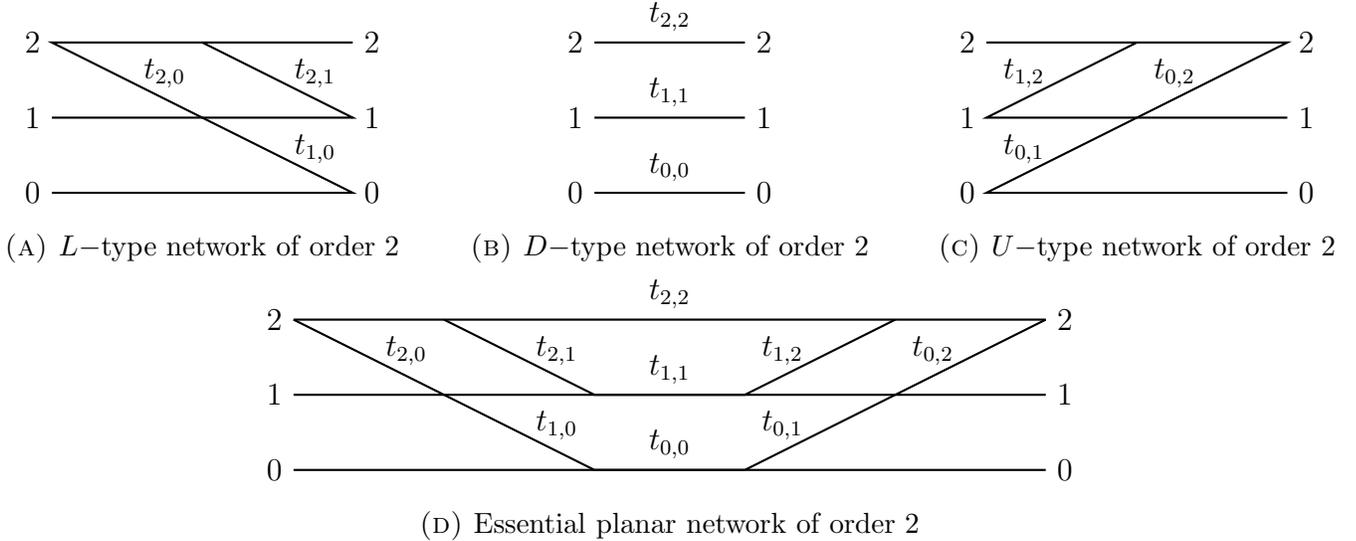
In this section, we constructed an $L_n$ lower triangular matrix obtained from an $L-$type network of order $n$, and similarly, the $D-$type and $U-$type networks were used to recover the entries of a diagonal matrix $D_n$ and an upper triangular matrix $U_n$, respectively. In addition, we use the fact that the concatenation of these networks is equivalent to computing the product of their corresponding weight matrices~\cite{Fallat, Fomin} to abtain the main theorem of this section. 
\begin{Theorem}\label{A=LDU}
The $LDU$ decomposition of a totally positive matrix $A$ can be recovered by decomposing the essential planar network associated with $A$ into an $L-$type, a $D-$type and a $U-$type networks, respectively.
\end{Theorem}
The proof of this theorem follows directly from Corollaries~\ref{Ln_unitlower},~\ref{Un_unitupper}, Lemma~\ref{Dn_diagonal} and the result on cancatenation of planar networks~\cite{Fallat, Fomin}.

\section{Formulae for $L_n$ and $U_n$}\label{formulas}
In this section we describe how we can compute the entries of the lower triangular matrix $L_n$ and the upper triangular matrix $U_n$ presented in the previous section. We begin with a recursive formula for $L_n$. To clarify the notation, $I_n$ is the $n\times n$ identity matrix, $\textbf{0}_n$ is the column vector of $n$ zeros, and $L_n^{-1}$ is the inverse of $L_n$, whose existence and complete description is presented in the following section.
\begin{Proposition}\label{recLn}
If $L_n$ is defined as in Definition~\ref{Ln}, then $L_{n+1}$ can be computed recursively in the following way
\begin{itemize}
\item $L_0=(1)$
\item $L_{n+1}=F_{L_n}(L_n\oplus1)$, where $F_{L_n}=\left(\begin{array}{cc}
I_{n+1}                 &   \textbf{0}_{n+1}\\
-L_{n+1}^{-1}[n+1,j]_{0\leq j\leq n}  &   1
\end{array}\right)$
\end{itemize}
\end{Proposition}
\begin{proof}
We use induction to prove the proposition. The base case is trivial. We consider the inductive step, $L_{n+1}=F_{L_n}(L_n\oplus1)$. We begin by expanding the product $F_{L_n}(L_n\oplus1)$ as follows:
\begin{align*}
F_{L_n}(L_n\oplus1)&=\left(\begin{array}{cc}
I_{n+1}                 &   \textbf{0}_{n+1}\\
-L_{n+1}^{-1}[n+1,j]_{0\leq j\leq n}  &   1
\end{array}\right)\left(\begin{array}{cc}
L_n                         &   \textbf{0}_{n+1}\\
\textbf{0}_{n+1}^\intercal  &   1
\end{array}\right)\\
&=\left(\begin{array}{cc}
L_n                         &   \textbf{0}_{n+1}\\
-L_{n+1}^{-1}[n+1,j]_{0\leq j\leq n}L_n   &   1
\end{array}\right)
\end{align*}
First, we show that $L_{n+1}[i,j]_{0\leq i,j\leq n}=L_n$. This can be best described combinatorially, where a path $p\in\mathcal{P}_{i,j}^{L_{n+1}}$ ($i\leq n$) is equivalent to a horizontal step $[(0,i),(1,i)]$ followed by an $\mathcal{L}_n^1-$path $p'$. Notice that a fall step $[(0,i),(1,i-1)]$ is not allowed since $i<n+1$ (see Definition~\ref{L_Path}). It follows that $\omega(p)=\omega(p')$, and our description clearly shows a bijection. So, indeed $L_{n+1}[i,j]=L_n[i,j]$ for $i,j\in\{0,\ldots,n\}$.\\
The rightmost column is trivial since $L_{n+1}$ is unit lower triangular by Corollary~\ref{Ln_unitlower}.
Finally, we need to show that $L_{n+1}[n+1,j]_{0\leq j\leq n}=-L_{n+1}^{-1}[n+1,j]_{0\leq j\leq n}L_n$. We begin by the simple fact that
\[L_{n+1}^{-1}L_{n+1}=I_{n+1}\]
Now, we rewrite it in block matrix notation, where $L_{n+1}^{-1}[i,j]_{0\leq i,j\leq n}=L_n^{-1}$ to guarantee the product is the identity matrix.
\[\left(\begin{array}{cc}
L_n^{-1}                &   \textbf{0}_{n+1}\\
L_{n+1}^{-1}[n+1,j]_{0\leq j\leq n}   &   1
\end{array}\right)\left(\begin{array}{cc}
L_n                 &   \textbf{0}_{n+1}\\
L_{n+1}[n+1,j]_{0\leq j\leq n}    &   1
\end{array}\right)=\left(\begin{array}{cc}
I_{n+1}                     &   \textbf{0}_{n+1}\\
\textbf{0}_{n+1}^\intercal  &   1
\end{array}\right)\]
Then, consider the product of the second row with the first column which produce the following equation
\[L_{n+1}^{-1}[n+1,j]_{0\leq j\leq n}L_n+L_{n+1}[n+1,j]_{0\leq j\leq n}=\textbf{0}_{n+1}^\intercal\] which, if rearranged, will show that
\[L_{n+1}[n+1,j]_{0\leq j\leq n}=-L_{n+1}^{-1}[n+1,j]_{0\leq j\leq n}L_n\]
\end{proof}
In a similar fashion, we define a recursive formula for $U_n$ as follows:
\begin{Proposition}\label{recUn}
If $U_n$ is defined as in Definition~\ref{Un}, then it follows the following recursion
\begin{itemize}
\item $U_0=(1)$
\item $U_{n+1}=(U_n\oplus1)F_{U_n}$, where $F_{U_n}=\left(\begin{array}{cc}
I_{n+1}                     &   -U_{n+1}^{-1}[i,n+1]_{0\leq i\leq n}\\
\textbf{0}_{n+1}^\intercal  &   1
\end{array}\right)$
\end{itemize}
\end{Proposition}
Since the proof follows the same ideas as Proposition~\ref{recLn}, it will be omitted.\\
The following lemma sets the stage for the main result of this section. More specifically, we provide a closed-form formula to compute the entries of $L_n$ ane $U_n$, respectively.

\begin{Lemma}\label{wghts_dcrsng}
For all $n\in \mathbb{N}$ and $0\leq i,j\leq n$, let $p\in\mathcal{P}_{i,j}^{L_n}$ and $\pi_1,\pi_2\in p$ such that $\omega(\pi_1)=t_{y_1,s_1}$ and $\omega(\pi_2)=t_{y_2,s_2}$ where $\omega$ is defined as in Definition~\ref{Ln}. If $y_1<y_2$, then $s_1\geq s_2$.
\end{Lemma}
\begin{proof}
From Definition~\ref{Ln}, we conclude that there are $x_1,x_2\in \{0,\ldots,n-1\}$ such that $\pi_1=[(m+x_1,y_1),(m+x_1+1,y_1-1)]$ and $\pi_2=[(m+x_2,y_2),(m+x_2+1,y_2-1)]$. It follows from the same definition that $s_1=x_1+y_1-n$ and $s_2=x_2+y_2-n$. Consequently, $s_1-s_2=(x_1-x_2)+(y_1-y_2)$. Given that $y_1<y_2$ and from Lemma~\ref{y_dcrsng}, $x_1>x_2$, which means $x_1-x_2>0$ and $y_1-y_2<0$. Nevertheless, we know that each step $x$ increases by $1$ while $y$ decreases by $1$ or remains the same. It is obvious then that $x_1-x_2\geq -(y_1-y_2)$. In other words, $s_1-s_2\geq 0$.
\end{proof}
In the following theorem, we define $Q_{i,j}^I$, for $i>j$, to be the set of increasing sequences of length $i-j$, and if $\alpha\in Q_{i,j}^I$, then $0\leq\alpha_r\leq i-r$. Let $\varepsilon$ be the empty sequence, and extend the definition such that $Q_{i,i}^I=\{\varepsilon\}$, and $Q_{i,j}^I=\phi$, the empty set for $i<j$.
\begin{Theorem}\label{Ln_frmla}
If $L_n$ is defined as in definition~\ref{Ln}, then
\begin{equation}\label{Ln_formula}
L_n[i,j]\,=\,\sum_{\alpha\in Q_{i,j}^I}\prod_{r=j}^{i-1}t_{r+1,\alpha_{i-r}}
\end{equation}
\end{Theorem}
\begin{proof}
If $i<j$, then $Q_{i,j}^I=\phi$. Thus, the sum is empty and $L_n[i,j]=0$.
If $i=j$, the sum is over the empty sequence $\varepsilon$ only. It follows that $\displaystyle L_n[i,j]=\prod_{r=j}^{i-1}t_{r+1,\varepsilon_{i-r}}=1$ since it is an empty product. These two results agree with Corollary~\ref{Ln_unitlower}.
Finally, we need to show that (\ref{Ln_formula}) holds for the case $i>j$. 
The proof relies on constructing a bijection between $Q_{i,j}^I$ and $\mathcal{P}_{i,j}^{L_n}$ such that if $\alpha\mapsto p$, then $\displaystyle\omega(p)=\prod_{r=j}^{i-1}t_{r+1,\alpha_{i-r}}$. Define $f(\alpha)=[(k,y_{_k})]_{0\leq k\leq n}$ such that
\[y_{_k}=\left\{\begin{array}{cl}
i   &   \mbox{, if $0\leq k\leq\alpha_1+n-i$}\\
i-r &   \mbox{, if $\alpha_r+n-(i-r)\leq k\leq\alpha_{r+1}+n-(i-r)$ for $r\in\{1,\ldots,i-j-1\}$}\\
j   &   \mbox{, if $\alpha_{i-j}+n-j\leq k\leq n$}
\end{array}\right.\]
First, we show $f$ is well defined. By definition of $Q_{i,j}^I$, we know that $0\leq\alpha_1\leq i-1$ and $\alpha_r\leq\alpha_{r+1}\leq i-(r+1)$. It follows that $0\leq n-i\leq\alpha_1+n-i\leq n-1$ and $\alpha_r+n-(i-r)\leq\alpha_{r+1}+n-(i-r)\leq n-1$ for $r\in\{1,\ldots,i-j-1\}$. By letting $r=i-j-1$, we get $\alpha_{i-j}+n-j\leq n$. Combining these inequalities yields,
\[0\leq\alpha_1+n-i<\alpha_1+n-(i-1)\leq\alpha_2+n-(i-1)<\alpha_2+n-(i-2)\leq\ldots\alpha_{i-j}+n-(j+1)<\alpha_{i-j}+n-j\leq n\]
which shows that by fixing $k$ there is a unique $y_{_k}$. It also proves that $y_0=i$, $y_n=j$, and if $k_1<k_2$, then $y_{_{k_1}}\geq y_{_{k_2}}$. It is easy to see from these that if $y_{_k}=s$ for $s\in\{j,\ldots,i\}$, then $y_{_{k+1}}\in\{s,s-1\}$. It remains to show that if $y_{_{k+1}}=y_{_k}-1$, then $y_{_k}\geq n-k$. Assume $y_{_{k+1}}=s$ for $s\in\{j,\ldots,i-1\}$. From the definition, we know that $k+1=\alpha_{i-s}+n-s$. Therefore, $n-k=\alpha_{i-s}+s+1$. However, $y_{_k}=s+1$ and $\alpha_{i-s}\geq 0$. As a consequence, $y_{_k}\geq n-k$. Now, we can conclude that $f(\alpha)\in\mathcal{P}_{i,j}^{L_n}$.\\
To show that $f$ is surjective, assume $p\in\mathcal{P}_{i,j}^{L_n}$ and $p=[(k,y_{_k})]$. Define $K$ to be the set of indices $k$ at which there is a fall step. More precisely, $K=\{k:k<n\text{ and }y_{_{k+1}}=y_{_k}-1\}$. Clearly, $K$ has $i-j$ elements, so call them $k_1,\ldots,k_{i-j}$ in ascending order. Now, define $\beta$ such that,
\begin{equation}\label{beta}
\beta_r=k_r+i-r+1-n.
\end{equation}
$k_r$ is strictly increasing, so $k_r-r$ is increasing. It follows that $\beta_r$ is increasing. Since $k_r<n$, $k_r+1-n\leq0$, and, consequently, $\beta_r\leq i-r$. Finally, we know from condition $(iii)$ in Definition~\ref{L_Path} that $k_r\geq n-y_{_{k_r}}$. Also, by definition, $y_{_{k_1}}=i$ and $y_{k_{r+1}}=y_{_{k_r}}-1$, which imply that $y_{_{k_r}}=i-(r-1)$. This implies that $k_r\geq n-i+r-1$, which implies directly that $\beta_r\geq0$. To sum up the results, we showed that $\beta$ is increasing and $0\leq\beta_r\leq i-r$. Therefore, $\beta\in Q_{i,j}^I$ and $f(\beta)$ is $p$.\\
To show that $f$ is injective, assume two distinct sequences $\alpha,\beta\in Q_{i,j}^I$. Define $D$ to be the set of indices at which they differ. More precisely, $D=\{1\leq r\leq i-j:\alpha_r\neq\beta_r\}$ which is clearly non-empty. Take $r_o$ to be the minimum of $D$, and assume $f(\alpha)=[(k,y_{_k}^\alpha)]$ and $f(\beta)=[(k,y_{_k}^\beta)]$. Since $\alpha_{r_o}\neq\beta_{r_o}$, assume, without loss of generality, that $\alpha_{r_o}<\beta_{r_o}$. Let $k_c$ be $\alpha_{r_o}+n-(i-r_o)$, so $k_c<\beta_{r_o}+n-(i-r_o)$. Therefore, $y^\alpha_{k_c}=i-r_o$ while $y^\beta_{k_c}=i-r_o-1$. Consequently, $f(\alpha)\neq f(\beta)$, and hence the injectivity of $f$.\\
Therefore, $f$ is a bijection between $Q_{i,j}^I$ and $\mathcal{P}_{i,j}^{L_n}$. It remains to show that $\displaystyle\omega(f(\alpha))=\prod_{r=j}^{i-1}t_{r+1,\alpha_{i-r}}$. Once more, consider the set $K=\{k:k<n\text{ and }y_{_{k+1}}=y_{_k}-1\}$. Let $\pi_k$ be the edge $[(k,y_{_k}),(k+1,y_{_k}+1)]$ in $f(\alpha)$, for a fixed $0\leq k<n$. It is clear from Definition~\ref{Ln} that
\[\omega(\pi_k)=\left\{\begin{array}{cl}
t_{y_{_k},k+y_{_k}-n} &   \text{, if }k\in K\\
1               &   \text{, if }k\notin K
\end{array}\right.\]
Since $\displaystyle\omega(f(\alpha))=\prod_{k=0}^{n-1}\omega(\pi_k)$, we deduce that $\displaystyle\omega(f(\alpha))=\prod_{k\in K}t_{y_{_k},k+y_{_k}-n}$. Using the order described on $K$, we rewrite the last formula as $\displaystyle\omega(f(\alpha))=\prod_{r=1}^{i-j}t_{y_{_{k_r}},k_r+y_{_{k_r}}-n}$. However, we have already proved that $y_{_{k_r}}=i-(r-1)$. Therefore, if we substitute the value of $y_{_{k_r}}$ and apply the transformation $s=i-r$, we reach the formula $\displaystyle\omega(f(\alpha))=\prod_{s=j}^{i-1}t_{s+1,k_{i-s}+s+1-n}$. Using (\ref{beta}), we know that $\alpha_{i-s}=k_{i-s}+i-((i-s)-1)-n=k_{i-s}+s+1-n$. Thus, substituting back in the last formula and replacing $s$ with $r$, we conclude that $\displaystyle\omega(f(\alpha))=\prod_{r=j}^{i-1}t_{r+1,\alpha_{i-r}}$. Now, we use Definition~\ref{Ln}, to write
\[L_n=\sum_{p\in\mathcal{P}_{i,j}^{L_n}}\omega(p)=\sum_{\alpha\in Q_{i,j}^I}\prod_{r=j}^{i-1}t_{r+1,\alpha_{i-r}}\]
\end{proof}

As the case of $L_n$, there is a closed-form formula for the entries of the matrix $U_n$ in terms of the $t_{i,j}$ parameters. This is highlighted in the following Theorem.

\begin{Theorem}\label{Un_frmla}
If $U_n$ is defined as in definition~\ref{Un}, then
\begin{equation}\label{Un_formula}
U_n[i,j]\,=\,\sum_{\alpha\in Q_{j,i}^I}\prod_{r=i}^{j-1}t_{\alpha_{j-r},r+1}
\end{equation}
\end{Theorem}
\begin{Proof}
If we let
\[y_{_k}=\left\{\begin{array}{cl}
j   &   \mbox{, if $j-\alpha_1\leq k\leq n$}\\
j-r &   \mbox{, if $j-r-\alpha_{r+1}\leq k\leq j-r-\alpha_r$ for $r\in\{1,\ldots,j-i-1\}$}\\
i   &   \mbox{, if $0\leq k\leq i-\alpha_{j-i}$}
\end{array}\right.\]
then a similar argument to the one in the proof of Theorem~\ref{Ln_frmla} is used.
\end{Proof}

\section{The Inverse of a Totally Positive Matrix}\label{inv}
The definition of totally positive matrices introduced in Section~\ref{sec:intro} ensures such matrices are invertible. In Section~\ref{LDU}, we used the fact that every totally positive matrix has a unique LDU decomposition and provided a combinatorial description of the $L$, $D$, and $U$. If the inverses of the $L$, $D$ and $U$ factors are known to be $L^{-1}$, $D^{-1}$ and $U^{-1}$, respectively, then the inverse of the totally positive matrix can be computed as the product $U^{-1}D^{-1}L^{-1}$. Also, since the product of the $L$, $D$ and $U$ factors yields a matrix that is invertible, then each of the factors is also invertible. In this section, we find a combinatorial description of the inverse of each of these factors, similar to the one described in Section~\ref{LDU}.
 
We introduce a new weight matrix, which we call $L^{-1}$, that can be recovered from an $L-$type network whose weights are defined as follows:
\begin{subequations}\label{invLn}
\begin{equation}
\text{for }j\neq 0\text{, }\omega([(m+i,j),(m+i+1,j-1)])=-t_{j,n-i-1}
\end{equation}
\begin{equation}
\omega([(m+i,j),(m+i+1,j)])=1
\end{equation}
\end{subequations}
for all $0\leq i,j\leq n$. It follows from Corollary~\ref{Ln_unitlower} that $L^{-1}$ is unit lower triangular.\\
In this section, we will introduce a formula for the entries of $L^{-1}$; which we will use to show that $L^{-1}$ is indeed the inverse of $L_n$.

\begin{Lemma}\label{wghts_incrsng}
For all $n\in\mathbb{N}$ and $0\leq i,j\leq n$, let $p\in\mathcal{P}_{i,j}^{L_n}$ and $\pi_1,\pi_2\in p$ such that $\omega(\pi_1)=-t_{y_1,s_1}$ and $\omega(\pi_2)=-t_{y_2,s_2}$ where $\omega$ is defined as in (\ref{invLn}). If $y_1<y_2$, then $s_1<s_2$.
\end{Lemma}
\begin{proof}
From (\ref{invLn}), we conclude that there are $x_1,x_2\in \{0,\ldots,n-1\}$ such that $\pi_1=[(m+x_1,y_1),(m+x_1+1,y_1-1)]$ and $\pi_2=[(m+x_2,y_2),(m+x_2+1,y_2-1)]$. It follows from the same definition that $s_1=n-x_1-1$ and $s_2=n-x_2-1$. Consequently, $s_1-s_2=x_2-x_1$. Given that $y_1<y_2$ and from Lemma~\ref{y_dcrsng}, $x_1>x_2$, which means $x_2-x_1<0$. Therefore, $s_1-s_2<0$.
\end{proof}
 In analogy to Theorem~\ref{Ln_frmla}, we give a closed-form formula for $L^{-1}$. To prove this, define $Q_{i,j}^{SD}$, for $i>j$, to be the set of strictly decreasing sequences of length $i-j$, and if $\alpha\in Q_{i,j}^{SD}$, then $0\leq\alpha_r\leq i-r$. Also, we extend the definition such that $Q_{i,i}^{SD}=\{\varepsilon\}$, the empty sequence for $i=j$ and $Q_{i,j}^{SD}=\phi$, the empty set for $i<j$.
\begin{Theorem}\label{invLn_frmla}
If $L^{-1}$ is the weight martix of the $L-$type network of order $n$ whose weights are defined as in (\ref{invLn}), then
\begin{equation}\label{invLn_formula}
L^{-1}[i,j]\,=\,(-1)^{i-j}\sum_{\alpha\in Q_{i,j}^{SD}}\prod_{s=j}^{i-1}t_{s+1,\alpha_{i-s}}
\end{equation}
\end{Theorem}
\begin{proof}
If we let
\[y_{_k}=\left\{\begin{array}{cl}
i   &   \mbox{, if $0\leq k\leq n-\alpha_1-1$}\\
i-r &   \mbox{, if $n-\alpha_r\leq k\leq n-\alpha_{r+1}-1$ for $r\in\{1,\ldots,i-j-1\}$}\\
j   &   \mbox{, if $n-\alpha_{i-j}\leq k\leq n$}
\end{array}\right.\]
then a similar argument to the one in the proof of Theorem~\ref{Ln_frmla} is used.
\end{proof}
We observe that $L^{-1}$ shares quite similar characteristics as the ones of the lower triangular matrix $L_n$ presented in Section~\ref{LDU}. It is not surprising to see that $L^{-1}$, the way defined, is the inverse of $L_n$ defined in Definition~\ref{Ln}.
\begin{Theorem}\label{Ln_invLn}
If $L^{-1}$ is the weight martix of the $L-$type network of order $n$ whose weights are defined as in (\ref{invLn}), then $L_nL^{-1}\,=\,I_{n+1}$, the $(n+1)\times(n+1)$ identity matrix.
\end{Theorem}
\begin{proof}
From Corollary~\ref{Ln_unitlower}, we know that both $L_n$ and $L^{-1}$ are unit lower triangular, and thus, their product is unit lower triangular as well. It remains to show that the dot product of row $i$ of $L_n$ with column $j$ of $L^{-1}$ is $0$, whenever $i>j$. To show that, we let $A_n=L_n L^{-1}$ and expand the product as $A_n[i,j]=\sum_{k=0}^n{L_n[i,k]L^{-1}[k,j]}$ which can be simplified, using Theorems~\ref{Ln_frmla} and~\ref{invLn_frmla}, as
\begin{equation}\label{expansion}
A_n[i,j]=\sum_{k=j}^i T_k\text{ , where }T_k=(-1)^{k-j}\sum_{\alpha\in Q_{i,k}^I}\sum_{\beta\in Q_{k,j}^{SD}}\prod_{r\in\{k\ldots i-1\}}t_{r,\alpha_{i-r}}\prod_{s\in\{j\ldots k-1\}}t_{s,\beta_{k-s}}
\end{equation}
Here we observe that if $k<j$ (or $k>i$), then $L^{-1}[k,j]=0$ (or $L_n[i,k]=0$) by lemma~\ref{Ln_unitlower}. Then, we define two more sums, $T_k^\leq$ and $T_k^>$ as follows,
\begin{subequations}\label{sum_parts}
\begin{equation}
T_k^\leq=\left\{\begin{array}{cl}
\displaystyle(-1)^{k-j}\sum_{\alpha\in Q_{i,k}^I}\sum_{\substack{\beta\in Q_{k,j}^{SD}\\\alpha_{i-k}\leq\beta_1}}\prod_{r\in\{k\ldots i-1\}}t_{r,\alpha_{i-r}}\prod_{s\in\{j\ldots k-1\}}t_{s,\beta_{k-s}}
&  \text{, if } j<k<i\\
T_k & \text{, if } k=i\\
0   & \text{, if } k=j
\end{array}\right.
\end{equation}
\begin{equation}
T_k^>=\left\{\begin{array}{cl}
\displaystyle(-1)^{k-j}\sum_{\alpha\in Q_{i,k}^I}\sum_{\substack{\beta\in Q_{k,j}^{SD}\\\alpha_{i-k}>\beta_1}}\prod_{r\in\{k\ldots i-1\}}t_{r,\alpha_{i-r}}\prod_{s\in\{j\ldots k-1\}}t_{s,\beta_{k-s}}
&  \text{, if } j<k<i\\
0   & \text{, if } k=i\\
T_k & \text{, if } k=j
\end{array}\right.
\end{equation}
\end{subequations}
Clearly, $T_k=T_k^\leq+T_k^>$, where $T_k^\leq$ and $T_k^>$ have no common terms. It follows from (\ref{expansion}) that $\displaystyle A_n[i,j]=\sum_{k=j}^i(T_k^\leq+T_k^>)$. Rearranging the terms, the sum becomes $\displaystyle A_n[i,j]=T_j^\leq+\sum_{k=j}^{i-1}(T_k^>+T_{k+1}^\leq)+T_i^>$. From definition (\ref{sum_parts}), $T_j^\leq=T_i^>=0$. Thus, $\displaystyle A_n[i,j]=\sum_{k=j}^{i-1}(T_k^>+T_{k+1}^\leq)$. We will now show that $T_k^>=-T_{k+1}^\leq$ for all $j\leq k<i$ so that $A_n[i,j]=0$ as desired.\\
Since $\alpha_{i-k}>\beta_1$, define new sequences $\alpha'$ and $\beta'$ such that $\alpha'_r=\alpha_r$ for $r=1\ldots i-k-1$, $\beta'_1=\alpha_{i-k}$, and $\beta'_r=\beta_{r-1}$ for $r=2\ldots k+1-j$. Since $\alpha$ is increasing, $\alpha'_{i-k-1}\leq\beta'_1$.Therefore, $T_k^>$ can be rewritten as
\[T_k^>=(-1)^{k-j}\sum_{\alpha'\in Q_{i,k+1}^I}\sum_{\substack{\beta'\in Q_{k+1,j}^{SD}\\\alpha'_{i-k-1}\leq\beta'_1}}\prod_{r\in\{k+1\ldots i-1\}}t_{r,\alpha'_{i-r}}\prod_{s\in\{j\ldots k\}}t_{s,\beta'_{k+1-s}}\]
which differs from $T_{k+1}^\leq$ by a factor of $-1$. Therefore, $A_n=I_{n+1}$.
\end{proof}
Then, similar to the case of $L_2$, we can now present an example (in Figure~\ref{fig:invLexample}) of an $L-$type network of order $2$ with weights defined as in (\ref{invLn}), and the weight matrix of that network, namely $L_2^{-1}$.
\begin{figure}[ht]
\begin{subfigure}{0.35\textwidth}
\centering
\begin{tikzpicture}
\draw[black, thick]   (0,0)node[anchor=east]{$0$}--(4,0)node[anchor=west]{$0$}--(3.5,0.25)node[anchor=south]{$-t_{1,0}$}--(1.5,1.25)node[anchor=south]{$-t_{2,1}$}--(0,2)node[anchor=east]{$2$}--(4,2)node[anchor=west]{$2$};
\draw[black, thick]   (0,1)node[anchor=east]{$1$}--(4,1)node[anchor=west]{$1$}--(3.5,1.25)node[anchor=south]{$-t_{2,0}$}--(2,2);
\end{tikzpicture}
\caption{$L-$type network of order $2$}\label{fig:invLnet2}
\end{subfigure}
\begin{subfigure}{0.35\textwidth}
\centering
$\left(\begin{array}{ccc}
1               &           0           &   0\\
-t_{1,0}        &           1           &   0\\
t_{2,1}t_{1,0}  &   -(t_{2,0}+t_{2,1})  &   1
\end{array}\right)$
\caption{The weight matrix $L_2^{-1}$}\label{fig:invL2}
\end{subfigure}
\caption{Description of $L_2^{-1}$}\label{fig:invLexample}
\end{figure}
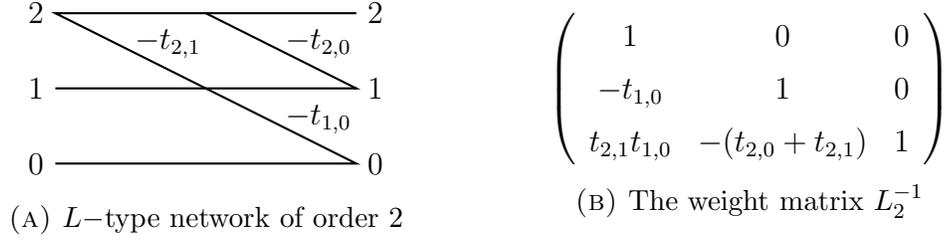

Similarly, we can define the inverse of $D_n$. All the entries on the diagonal of $D_n$ are positive, so it is invertible. The inverse, namely $D_n^{-1}$ is the diagonal matrix whose diagonal entries are the reciprocals of diagonal elements of $D_n$. Combinatorially, $D_n^{-1}$ can be recovered from a $D-$type network of order $n$ where the weights are defined as \[\omega([(0,i),(1,i)])=\frac{1}{t_{i,i}}\] for $i=0\ldots n$. The proofs that the combinatorial description results in a diagonal matrix $D_n^{-1}$ is similar to the proof of Lemma~\ref{Dn_diagonal}, so it will be skipped. The general form of $D_n^{-1}$ is
\[D_n^{-1}\,=\,\left(\begin{array}{cccc}
\frac{1}{t_{0,0}}   &   0                   &   \cdots  &   0\\
0                   &   \frac{1}{t_{1,1}}   &   \cdots  &   0\\
\vdots              &                       &   \ddots  &   \vdots\\
0                   &   0                   &   \cdots  &   \frac{1}{t_{n,n}}
\end{array}\right)\]

To find the inverse of $U_n$, we will mimic the same approach as the one for defining $L_n^{-1}$. We will define $U^{-1}$ to be the weight matrix of a $U-$type network whose weights are defined as follows:
\begin{subequations}\label{invUn}
\begin{equation}
\text{for }j\neq 0\text{, }\omega([(m+i,j-1),(m+i+1,j)])=-t_{i,j}
\end{equation}
\begin{equation}
\omega([(m+i,j),(m+i+1,j)])=1
\end{equation}
\end{subequations}
for all $0\leq i,j\leq n$. It follows from Corollary~\ref{Un_unitupper} that $U^{-1}$ is unit upper triangular.\\
Regarding $U^{-1}$, we obtain the following results.
\begin{Theorem}\label{invUn_frmla}
If $U^{-1}$ is the weight matrix of the $U-$type network of order $n$ whose weights are defined as in (\ref{invUn}), then
\begin{equation}\label{invUn_formula}
U^{-1}[i,j]\,=\,(-1)^{i-j}\sum_{\alpha\in Q_{j,i}^{SD}}\prod_{s=i}^{j-1}t_{\alpha_{j-s},s+1}
\end{equation}
\end{Theorem}
\begin{Proof}
If we let
\[y_{_k}=\left\{\begin{array}{cl}
j   &   \mbox{, if $\alpha_1+1\leq k\leq n$}\\
j-r &   \mbox{, if $\alpha_{r+1}+1\leq k\leq\alpha_r$ for $r\in\{1,\ldots,j-i-1\}$}\\
i   &   \mbox{, if $0\leq k\leq\alpha_{i-j}$}
\end{array}\right.\]
then a similar argument to the one in the proof of Theorem~\ref{Ln_frmla} is used.
\end{Proof}
\begin{Theorem}\label{Un_invUn}
If $U^{-1}$ is the weight matrix of the $U-$type network of order $n$ whose weights are defined as in (\ref{invUn}), then $U_nU^{-1}\,=\,I_{n+1}$, the $(n+1)\times(n+1)$ identity matrix.
\end{Theorem}
Expanding the product of $U_nU^{-1}$ as we did in the proof of Theorem~\ref{Ln_invLn}, proves this theorem.\\

It is worth noting that all the results of this paper can be easily extended to totally nonnegative matrices if the principal minors are nonzero. This is equivalent to saying the weights of the $D-$network, namely $t_{i,i}$ cannot be $0$ for all $0\leq i\leq n$; which guarantee the invertibility and, by Theorem~\ref{uniqueLU}, the uniqueness of the $LU$ factorization with $L$ unit lower triangular. Therefore, it has a unique $LDU$ factorization with $L$ unit lower triangular, $D$ diagonal and $U$ unit upper triangular.


\begin{thebibliography}{MM}

\frenchspacing
\renewcommand{\baselinestretch}{1}

\bibitem{Ando}
T. Ando, Totally Positive Matrices, {\em Linear Algebra Appl.} 90 (1987), 165--219.

\bibitem{Canto}
R. Cant\'{o}, P. Koev, B. Ricarte, and A. M. Urbano
LDU Factorization of Nonsingular Totally Nonpositive Matrices
{\em Siam J. Matrix Analy. Appl.} (2) 167 (2008), no. 1, 53--94.

\bibitem{Cryer}
C. W. Cryer, The LU-Factorization of Totally Positive Matrices,
in {\em Linear Algebra and Its Applications.}, 7 (1973), 83--92  

\bibitem{Fallat}
S. M. Fallat and C. R. Johnson, {\em Totally Nonnegative Matrices}, Princeton University Press, 2000.

\bibitem{Fomin}
S. Fomin and A. Zelvinsky,
Total Positivity: Tests and Parametrizations, Springer-Verlag New York. Volume 22, Number 1, 2000.

\bibitem{GantKrein}
F. R. Gantmacher and M. G. Krein, Sur les matrices compl\`{e}tement non n\'{e}gatives et oscillatoires,
{\em Compositio Math.} 4 (1937), 445--476.

\bibitem{Gant}
F. R. Gantmacher, {\em The Theory of Matrices}, Chelsea, New York, Volume 7, 1959.

\bibitem{GascaMicch}
M. Gasca and C. A. Micchelli, Eds., Total Positivity and its Applications, {\em Math. Appl.} 359,
Kluwer Academic Publishers, Dordrecht, The Netherlands, 1996.

\bibitem{GascaPena_Nev}
M. Gasca and J. M. Pe\~{n}a, Total positivity and Neville elimination, {\em Linear Algebra Appl.}, 44
(1992), 25--44.

\bibitem{GascaPena_factor}
M. Gasca and J. M. Pe\~{n}a, On factorizations of totally positive matrices, in {\em Total Positivity and its Applications},
Kluwer Academic Publishers, Dordrecht, The Netherlands, 1996, 109--130.

\bibitem{Karlin}
S. Karlin, Total positivity, {\em Stanford University Press}, Stanford, CA, 1968.

\end{thebibliography}
\end{document}